\documentclass{article}
\usepackage{amsmath,amssymb,latexsym,epsfig, amsthm,amscd}
\usepackage{palatino}
\usepackage{fancyhdr}
\input xy
\xyoption{all}
\usepackage{latexsym}

\newcommand{\bb}[1]{{\mathbb #1}}    % Bourbaki 

\newcommand{\bbR}{{\bb R}}

\newcommand{\bbZ}{{\bb Z}} 
\newcommand{\bbQ}{{\bb Q}}

\newcommand{\bbC}{{\bb C}}

\newcommand{\GL}{{\rm GL}}

\newcommand{\SL}{{\rm SL}}

\newcommand{\U}{{\rm U}} 

\newcommand{\SP}{{\rm Sp}}
\newcommand{\Sp}{\SP}

\newcommand{\Aut}{{\rm Aut}}
\newcommand{\Hom}{{\rm Hom}}

\newcommand{\tensor}{\otimes}

\newcommand{\id}{{\rm id}}

{ \par  \medskip  \noindent  {\bf Definition} \hspace{0.5em} }%
{\par \medskip }

\newenvironment {remark}%
{{\em Remark \hspace{0.5em}} }%
{\par \medskip }

\newtheorem{proposition}{Proposition}[section]
\newtheorem{definition1}[proposition]{Definition}
\newtheorem{theorem}[proposition]{Theorem}
\newtheorem*{theorem*}{Theorem}
\newtheorem{lemma}[proposition]{Lemma}
\newtheorem{corollary}[proposition]{Corollary}
{ \begin{definition1} \rm }%
{ \end{definition1}}

\newcounter{note}
\newcommand{\ra}{\rightarrow}
\newcommand{\lra}{\longrightarrow}
\newcommand{\frakD}{\mathfrak D}

\title{Virtually abelian K\"ahler and projective groups}
\author{Oliver Baues\footnote{baues@mathematik.uni-karlsruhe.de} \\
 \small Institut f\"ur Algebra und Geometrie\\
 \small  Universit\"at Karlsruhe\\
 \small  D-76128 Karlsruhe  \and 
 Johannes Riesterer\footnote{riesterer@mathematik.uni-karlsruhe.de} \\ 
 \small Institut f\"ur Algebra und Geometrie\\
\small  Universit\"at Karlsruhe\\
 \small  D-76128 Karlsruhe
 }

\begin{document}
\date{October 26, 2009}

\maketitle

\begin{abstract} \noindent
We   characterise  the virtually abelian groups which are 
fundamental groups of compact  K\"ahler manifolds and 
of smooth projective varieties. 
% In particular, 
We show that a virtually abelian group is K\"ahler if and
only if it is projective. In particular, this allows to describe the K\"ahler condition
for such groups in terms of integral symplectic representations. 
\end{abstract}

\setcounter{section}{0}

\section{Introduction}
A connection between K\"ahler geometry and topology is made by the topic 
of K\"ahler groups. A group is called K\"ahler, if it is isomorphic to the fundamental group of a compact K\"ahler manifold. Note that K\"ahler groups are finitely generated and even finitely presented, since the associated manifold is compact by assumption.  
Not all finitely presented groups are K\"ahler groups. Well known restrictions stem from Hodge-theory, see \cite{ABKT} for various results.  

More restrictively, a compact complex manifold is called projective if it embeds holomorphically as a subvariety of complex projective space. A projective manifold is also 
a K\"ahler manifold, and its fundamental group is  called a projective K\"ahler group. 
It is widely believed that the class of K\"ahler and projective groups coincide.
An important remark is due to J.P. Serre \cite[prop.\ 15]{serre}:
\begin{theorem*}
\label{serre}
Every finite group $G$ is the fundamental group of a smooth projective variety $X_G$. 
\end{theorem*}

We would like to understand how the class of K\"ahler groups behaves with respect to finite extension of groups. Since not much seems to be known about 
this particular question, the aim of this paper is to study the situation in a specific
and accessible context: A group is called  virtually abelian if it is an extension of an abelian group by a finite group. We shall examine the following question: \emph{Which virtually abelian groups are K\"ahler and which are projective}? 

\subsection{K\"ahler and projective groups}
Let $\Gamma$ be a virtually abelian group.  
Since we also assume  $\Gamma$  to be finitely generated,
there is an exact sequence
\begin{equation}
\label{eq:fes} % \label{intro_1}
\xymatrix@1{
  0\ar[r] & \mathbb{Z}^n \ar[r]%^{i} 
  & \Gamma \ar[r]%^{p} 
  & G\ar[r] & 1}\; ,
\end{equation}
where $G$ is a finite group. Via conjugation in $\Gamma$, 
the extension %  \eqref{eq:fes} % 
defines the characteristic representation $\mu: G \to \GL(n,\mathbb{Z})$.  

Recall that a complex structure on $\bbR^n$ is a linear map $J: \bbR^n \ra \bbR^n$ with 
$J^2 = -\id$.  It is called $\mu$-invariant, for a representation $\mu: G \ra \GL(n,\bbR)$, if $\mu(g)J=J\mu(g)$, for all $g \in G$. 
We shall say that the virtualy abelian group $\Gamma$ admits a complex structure, if there exists
a complex structure on $\bbR^n$, which is invariant by the characteristic
representation $\mu$. 

\begin{theorem} \label{thm:Kaehler}
The group $\Gamma$ is K\"ahler if and only if it admits a complex structure.
\end{theorem}
Remark, as follows from elementary Hodge theory, if the torsionfree abelian group 
$\bbZ^n$ is K\"ahler then $n=2k$ is even. For a  K\"ahler group 
$\Gamma$,  $n$ is thus even in the above exact sequence. Theorem \ref{thm:Kaehler}
gives a more precise condition. It also shows that the K\"ahler condition
for $\Gamma$ depends only on the characteristic representation $\mu$.
In particular, it does not depend on the characteristic cohomology class of 
the extension \eqref{eq:fes}. 

We proceed  now to provide  a characterization of 
virtually abelian projective groups. For this, recall that
a two-form $\omega \in \bigwedge^2 (\bbR^n)^*$ is called a polarisation 
for $J$ (with respect to the lattice $\bbZ^n$)
if the following conditions are satisfied: 
\begin{enumerate}
\item 
$\omega(J \cdot, J \cdot ) = \omega(\cdot, \cdot)$, 
\item $g= \omega(J \cdot, \cdot)$
is positive definite,
\item 
$\omega(\bbZ^n, \bbZ^n) \subseteq \bbZ$. 
\end{enumerate}
Accordingly, if there exists a polarisation then the complex 
structure $J$ is called polarisable. 
%\begin{remark}
%If $J$ is a polarisable complex structure for $\Gamma$, then $\omega' := \sum_{g \in G} \mu(g)^*\omega$ is a $\mu(G)$-invariant polarisation for $J$.
% 
%\end{remark}

\begin{theorem} \label{thm:projective}
The group $\Gamma$ is projective K\"ahler if and only if it admits a polarisable complex structure.
\end{theorem}

We call an integral representation symplectic if it
preserves a non-degenerate alternating 
two-form $\omega \in \bigwedge^2 (\bbZ^n)^*$. 
We can add the following characterisation of the projectivity condition
for $\Gamma$: 

\begin{theorem} \label{thm:symplectic}
The following conditions are equivalent:
\begin{enumerate}
\item $\Gamma$ admits a polarisable complex structure.
\item The characteristic representation $\mu$ for $\Gamma$ is 
$\SL_{n}(\bbZ)$-conjugate to an integral symplectic representation. 
\end{enumerate}
%\item There exists $g \in \GL_{n}(\bbQ)$ such that $g \mu(G) g^{-1} \leq \Sp(n,\bbQ)$.

%Here $\Sp(n,\bbQ)$ denotes the linear isometry group of the standard 
%non-degenerate alternating two-form on $\bbQ^n$, $n=2k$. 
\end{theorem} 
\noindent Note that condition 2.\ does not refer to the existence of
a complex structure for $\Gamma$. 
The following consequence (see Theorem \ref{thm:algequivalence}) 
is implied:

\begin{theorem} \label{thm:kequivp}
The  group $\Gamma$ is K\"ahler if and only if it is projective K\"ahler.
\end{theorem}

In particular, we obtain a characterisation of the K\"ahler condition
in terms of integral symplectic representations. 

\begin{corollary} \label{cor:kequivp2}
The  group $\Gamma$ is K\"ahler if and only if the characteristic representation 
$\mu$ is conjugate to an integral symplectic representation. 
\end{corollary}

\paragraph{Fields of definition}
We may further refine the projectivity condition for $\Gamma$, by 
taking into account the possible fields of definition for projective varieties
with fundamental group $\Gamma$.  In this direction, we remark: 

\begin{theorem} \label{thm:numberfields}
If the virtually abelian group $\Gamma$ is K\"ahler 
then there exists a smooth algebraic subvariety $X$ of ${P}^N \bbC$, 
which is defined over an algebraic number field  and has 
fundamental group $\Gamma$.
\end{theorem}

\subsubsection{Flat  and aspherical K\"ahler groups}
Flat K\"ahler manifolds are constructed 
as quotients 
$M =\bbC^k/\Gamma$, where $\Gamma \leq \bbC^k \rtimes U(k)$ is a discrete and torsionfree subgroup of the isometry group of  $\bbC^k$. By a famous theorem of Bieberbach (see, for example, \cite{wolf}),  $\Lambda = \Gamma \cap \bbC^k$ is a lattice in $\bbC^k$ and of finite index in $\Gamma$. In particular, $\Gamma$ is a virtually abelian group. 
% The previous results also imply: 
Conversely, if $\Gamma$ is a
torsionfree virtually abelian K\"ahler group, it may as well be represented as the
fundamental group of a flat K\"ahler manifold. In fact, we obtain: 

\begin{corollary} \label{cor:flatKaehler}
Let $\Gamma$ be torsionfree. The  following conditions are equivalent:  
\begin{enumerate}
\item $\Gamma$ is a torsionfree K\"ahler group.
\item $\Gamma$ is the fundamental group of a flat K\"ahler manifold. 
\item $\Gamma$ is the fundamental group of a  flat  
K\"ahler manifold which is projective. 
\end{enumerate}
\end{corollary}

The following concept is closely related. 
Recall that a manifold is called aspherical if its universal covering
space is contractible. Henceforth, 
a K\"ahler group is called {\em aspherical} if it is the fundamental 
group of a compact aspherical K\"ahler manifold. Combining with 
results from \cite{bauescortes}, we can state: 

\begin{corollary} \label{cor:aspKaehler}
The  following conditions are equivalent:  
\begin{enumerate}
\item $\Gamma$ is an aspherical  K\"ahler group.
\item $\Gamma$ is a K\"ahler group which is torsionfree.
\end{enumerate}
\end{corollary}

\subsection{Deformation of complex representations to integral symplectic representations}
The above results % obviously imply 
translate into the following representation 
theoretic result:

\begin{theorem} \label{thm:algequivalence} 
Let $G$ be a finite group and $V$  a finite-dimensional 
$\bbQ[G]$-module. Then the following conditions are equivalent:
\begin{enumerate}
\item $V \tensor \bbR$ has a $G$-invariant 
complex structure $J$. 
\item $V$ is isomorphic to a symplectic $\bbQ[G]$-module $W$. 
\item There exists a $G$-invariant complex structure 
$J'$ on $V \tensor \bbR$,  which is polarisable.  %  by some $G$-invariant $\omega \in \bigwedge^2 V^*$. 
\end{enumerate}
($W$ is called symplectic if it admits a non-degenerate $G$-invariant
two-form.)
\end{theorem}

The method of our proof allows to show slightly more. 
Let  $\mathfrak{D}(G)$ denote the set of $G$-invariant complex structures
on $V \tensor \bbR$. Then the implication from 1.\ to 3.\ in Theorem \ref{thm:algequivalence} will be implied by

\begin{theorem} \label{thm:deformation} 
Let $G$ be a finite group and $V$  a finite-dimensional 
$\bbQ[G]$-module such that $V \tensor \bbR$ has a $G$-invariant 
complex structure $J$. Then $J$ may be continuously deformed in 
$\mathfrak{D}(G)$ to a $G$-invariant complex structure $J'$,  
which is polarisable. Moreover, the subset $\mathfrak{D}(G)_{p}$ 
of polarisable elements in   $\mathfrak{D}(G)$ is dense in $\mathfrak{D}(G)$.
\end{theorem}

\subsubsection{Deformation of complex structures on torus quotients}
A complex torus $T = \bbC^k/\Lambda$ which is projective is 
an abelian variety. By the usual dictionary between complex tori (resp.\ abelian varieties) and complex structures (resp.\ polarisable complex structures), the previous
results can be translated into statements about complex tori with a holomorphic
$G$-action: \\

Let $G$ be a finite group of automorphisms of a complex 
torus $T$, which acts freely on $T$.  Let $X_{G} = T/G$ be the quotient 
complex manifold. Then we may view $\mathfrak{D}(G)$ as the deformation
space of complex structures for $X_{G}$. Thus, 
Theorem  \ref{thm:deformation} implies

\begin{proposition} \label{prop:deformation2}
The complex structure on $X_{G}$ may be continuously deformed 
to a projective structure.
\end{proposition}

We call a complex manifold \emph{deformation rigid} 
if it does not admit any non-trivial 
continuous deformation of its complex structure. 
The following application 
concerning K\"ahler rigidity was suggested to us by F.\ Catanese:

\begin{corollary} \label{cor:rigidity} 
%Let $G$ be a finite group of automorphisms of a complex 
%torus $T$. If the action of $G$ on $T$ is deformation rigid
%then $T$ is an abelian variety. In particular, 
Let $X$ be a complex variety which is a finite quotient of
a complex torus. If $X_{G}$ is deformation rigid then $X_{G}$ is 
a quotient of an abelian variety, in particular,  $X_{G}$ is projective.  
\end{corollary}

%We call the holomorphic action of 
%$G$ on a complex torus $T$ deformation rigid if 
%$\Hom(G, GL(k,\bbC)/ \GL(k,\bbC)$ is discrete. 

%Torsionfree virtually abelian groups occur naturally as fundamental 
%groups of 

\subsection{Notes}
Theorem \ref{thm:Kaehler} appears as the main result of
the diploma thesis \cite{riesterer}. The analogue of Theorem \ref{thm:Kaehler}
concerning flat K\"ahler manifolds, and the equivalence 
of 2) and 3) in Corollary \ref{cor:flatKaehler} are the main
theorems in Johnson's paper \cite{Johnson}. 
%Note that Corollary \ref{cor:algequivalence} is an immediate\footnote{As 
%kindly pointed out to us by E.\ Looijenga.}  consequence of our
%Theorem \ref{thm:symplectic}. 

Serre \cite{Serre2} constructed certain examples of
subvarieties $X$ of complex projective space, which
are defined over algebraic number fields, and whose 
fundamental groups are semi-direct extensions of $\bbZ^n$ by cyclic
groups.  Our Theorem \ref{thm:numberfields} extends this 
construction without the restriction on the fundamental
group. In \cite{Serre2},  the construction is used to exhibit 
Galois-conjugate examples $X$ and $X'$
of this type with non-isomorphic fundamental group. 

The equivalence of 1.\ and 3.\ in 
Theorem \ref{thm:algequivalence} is also contained 
in \cite{Johnson}. The proof is based on a different 
approach, using the theory of complex multiplication.  
However, we will pick up this idea in the proof of Theorem 
\ref{thm:numberfields}. % See section 5 of this paper.

In \cite{popovzarhin} irreducible symmetry groups of complex tori are 
investigated from the point of view of complex reflection groups.
In particular, examples of non-\- polarisable tori with a 
non-trivial symmetry group are considered. 

By Corollary \ref{cor:kequivp2}, virtually abelian K\"ahler groups 
of rank $n=2k$ can be classified in principle by the determination of the 
finite subgroups of  $\Sp(2k, \bbZ)$. For $k=1$, this list
comprises the finite cyclic groups of order 2,3,4,6 (see \cite{zieschang}, for example).
Fujiki \cite{fujiki} classifies the finite automorphism groups of 
abelian surfaces, the three dimensional case is
covered in \cite{BGL}.
 
\subsection{Structure of the paper}
The proofs of Theorem \ref{thm:Kaehler} and Theorem \ref{thm:projective}
can be found in sections 2 and 3, see Proposition \ref{vak_prop_2} and Proposition 
\ref{prop:constructX}.
In section 4 we prove Theorem \ref{thm:algequivalence} (see in particular \S \ref{rationalpoints})
and Theorem \ref{thm:deformation}. 
The latter imply Theorem \ref{thm:symplectic} and Theorem \ref{thm:kequivp}. 
Section 5 is devoted to the proof of Theorem \ref{thm:numberfields}. \\

\vspace{0.5cm}
{\bf Acknowledgements} The first named author learned several years ago 
from A.\ Szczepa\'nski about Johnson's  and Rees' results on flat 
K\"ahler and projective manifolds. Fritz Grunewald informed us that 
he has characterised virtually abelian fundamental groups of projective
algebraic varieties in terms of integral symplectic representations, 
this question playing a role in joint work with Ingrid Bauer and Fabrizio 
Catanese. We wish to thank all three of them for various discussions. 
Part of this paper was written up
during a stay of the first named author at the research activity 
'Groups  in algebraic geometry', Centro Ennio di Georgi, Pisa, 2008. 
He wishes to thank the organisers for the hospitality and support.
Furthermore we wish to thank Eduard Loojenga
for a helpful comment, as well as Vicente Cort\'es, Evelina Viada 
and Enrico Leuzinger for conversations regarding the context of this paper.
We thank Stefan K\"uhnlein for commenting on a draft of section 5. \\

%%%%%%%%%%%%%%%%%%%%%%%%%%%%%%%%%%%%%%%%%%%%

\section{The Picard variety and induced Hodge structure on 
the fundamental group}
\label{hodge_structure}
In this section,  we prove that every virtually abelian K\"ahler group $\Gamma$ admits an invariant complex structure $J$. Furthermore we show, if $\Gamma$ is projective then it admits a polarised complex structure $J'$.  

\paragraph{Conventions concerning cohomology groups}
Let $X$ be a space, and $A$ an abelian group. Then
$H^j_{sing}(X,A)$ denotes the singular cohomology groups
with coefficients in $A$. Moreover, if $X$ is a manifold and $A$
is a subgroup of $\bbC$,
we let $H^j(X,A)$ denote the \emph{image} of 
$H^j_{sing}(X,A)$ in the de Rham cohomology group 
$H^j(X,\bbC) := H^j_{DR}(X,\bbC)$ under the de Rham isomorphism
$H^j_{sing}(X,\bbC) \ra H^j_{DR}(X,\bbC)$. Thus, in particular, 
$H^n(X, \mathbb{Z})$ denotes the full lattice in $H^n(X, \mathbb{R})$, which 
is the image of $H^n(X, \mathbb{Z})_{sing}$ under the coefficient homomorphism.  
For any continuous (respectively 
smooth) map $f: X \ra X'$, there is a well defined induced map on cohomology groups  $f^*: H^j(X', \mathbb{R}) \to H^j(X, \mathbb{R})$, such that $f^*(H^j(X', A)) \subset H^j(X, A)$. 

\subsection{Review of Picard varieties}
We start by reviewing the definition of Picard varieties.
Let $X$ be a K\"ahler manifold and  $H^1(X,\mathbb{C}) = H^{1,0}(X) \oplus H^{0,1}(X)$ the Hodge decomposition. 
% with respect to a complex structure $J$, i.e.  $H^{1,0}(X)$ is the eigenspace of $J$ for the eigenvalue $i$ and  $H^{0,1}(X)$ for $-i$.
Then there is an associated  isomorphism of real vector spaces 
\begin{eqnarray*}
&&\Re:   H^{0,1}(X) \to H^1(X,\mathbb{R}) , \\
&&\Re \, \omega  = \frac {\omega + \bar \omega}{2},  \; \; \omega \in H^{0,1}(X)  \; . 
\end{eqnarray*}
% is a real isomorphism and $H^{0,1}(X)$ is a complex vector space,  
Thus  $J:= \Re I \Re^{-1}$ is a complex structure on $H^1(X,\mathbb{R})$, where $I$ 
denotes multiplication with $i$ on the complex vector space $H^{0,1}(X)$.
The complex torus 
$$Pic^0(X):= (H^1(X,\mathbb{R}) / H^1(X, \mathbb{Z}),J)\; ,$$ is called the Picard variety of $X$. By construction,   $Pic^0(X)$ is biholomorphic to $H^{0,1}(X)/ H^1(X, \mathbb{Z})$, which yields an equivalent definition. For details, see \cite[11.11]{birken} or \cite[7.2.2]{voisin}.

\begin{lemma}
\label{vak_prop_1}
 $Pic^0$ is functorial, that is, if $f:X \to X'$ is a holomorphic map between K\"ahler manifolds, then $f$ induces a holomorphic map $f^*: Pic^0 (X') \to Pic^0 (X)$. 
\end{lemma}
\begin{proof}
Let $X, X'$ be K\"ahler manifolds and $f:X \to X'$ a holomorphic map. 
There is an induced map $f^*: H^1(X', \mathbb{R}) \to H^1(X, \mathbb{R})$ on the cohomology, such that $f^*(H^n(X, \mathbb{Z})) \subset H^n(X, \mathbb{Z})$. Hence,  $f^*$ induces  a homomorphism from $Pic^0(X')$ to $Pic^0(X)$. 

Now we show, that $f^* : Pic^0(X') \to Pic^0(X)$ is holomorphic.
For this consider the induced complex linear map   $f^*: H^1(X', \mathbb{C}) \to H^1(X, \mathbb{C})$, given by $f^*(\omega \otimes z):= f^*\omega \otimes z$. 
Since $f$ is holomorphic, it preserves the Hodge decomposition 
(see \cite{Wells}),  that is, 
% \begin{eqnarray*}
%  f^*({\Omega^{1,0}}') \subset \Omega^{1,0} \; , 
%  f^*({\Omega^{0,1}}') \subset \Omega^{0,1}  \; .
% \end{eqnarray*}
% For details see \cite{KN2}. Furthermore  
\begin{eqnarray*}
 f^*(H^{1,0}(X')) \subset H^{1,0}(X) \;  , \; \,  
 f^*(H^{0,1}(X')) \subset H^{0,1}(X)  \;   .
\end{eqnarray*}
%since $df^* = f^*d$ and $d = \partial + \bar{\partial}$. 
Let $J$ be the complex structure of $Pic^0(X)$ and $J'$ of $Pic^0(X')$ respectively. Since  $f^*: H^1(X',\mathbb{C}) \to  H^1(X,\mathbb{C})$ preserves the real structures, 
$\Re f^* = f^*\Re$. We deduce, that $f^*J' = Jf^*$. Thus $f^*$ is holomorphic.
\end{proof}

%\paragraph{Digression on projective manifolds}
%Now we take a closer look to the case, that $X$ is also a projective manifold. We use the following notations:

%A hermitian form $h$ on a complex vector space is as usually given by  $h = g + i\omega$,  where $\omega$ is an alternating form, such that $\omega(J_0x, J_0y)= \omega(x,y)$ and $g(x,y)= \omega(J_0x,y)$. 

%Recall that a projective manifold admits a K\"ahler metric such
%that the K\"ahler form $\omega$ has rational periods, i.e., 
%$[\omega] \in H^2(M,\mathbb{Q})$.
%%(here the 
%%singular cohomology group $H^2(M,\mathbb{Q})$ 
%%is viewed as a subspace of $H^2_{DR}(X,\bbR)$ via 
%%the de Rham isomorphism.)
%Conversely, by Kodaira's theorem, a compact K\"ahler manifold $(X,\omega)$ 
%is projective if $[\omega] \in H^2(M,\mathbb{Q})$.
%(For details,  see, for example,  \cite{Wells}.)

An abelian variety is a complex torus $\mathbb{C}^k / \Lambda$, which is also a projective manifold. By the criterion of Riemann \cite[p.35]{mumford}, this is the case, if and only if there exists a positive definite hermitian form $h$ on $\bbC^k$ such that 
its imaginary part $\omega:= \Im h$ is integral on $\Lambda \times \Lambda$. 
The form $\omega$ is called a polarisation for $\Lambda$. 

Now let $X$ be a projective manifold. Then there exists a polarisation $\omega$ for  $X$. This means, $\omega$ is a K\"ahler form, such that its cohomology class $[\omega]$ is integral. Then the Hodge-Riemann pairing
$$h(\alpha, \beta)= -2i \int_{X}\omega^{n-1} \wedge \alpha \wedge \bar{\beta} \; ,\;  \alpha, \beta \in H^{0,1}(X)\; ,$$ 
defines a positive definite hermitian form on $H^{0,1}(X)$. Moreover, its imaginary part, $\omega= \Im h$ 
 % where \Im h$ (\alpha, \beta)= \frac{1}{2i}(h(\alpha, \beta)- h(\beta, \alpha))$  
is integral on $H^1(X,\mathbb{Z})$. See \cite[7.2.2]{voisin} for discussion. 
% For details see \cite[7.3]{voisin}.
% For details see \cite[Lemma 11.11.4]{birken}. Hence we get 
In particular, it follows

\begin{lemma}
\label{moduli_prop_picab}
Let $X$ be a projective manifold. Then $Pic^0(X)$ is an abelian variety.
\end{lemma}

\subsection{Induced invariant Hodge structures on $\Gamma$}
Let $\Gamma$ be a virtually abelian K\"ahler group, and 
let $X$ be a compact K\"ahler manifold with $\pi_1(X)= \Gamma$.   
Let $\tilde{X}$ be the universal covering manifold of $X$. 
Associated to the exact sequence
\eqref{eq:fes}, with characteristic homomorphism $\mu:G \to\GL_{n}(\mathbb{Z})$,
we consider the holomorphic 
covering $$p: \, X_{1}  := \tilde{X}/ \mathbb{Z}^{n}\;  \lra  X \; . $$
% Here, $X_{1}$ is a compact K\"ahler manifold. 
Since this is a regular covering, $G= \Gamma / \mathbb{Z}^{n}$ identifies with the group of covering transformations of $p$. In particular, $G$ has an induced 
holomorphic action on the Picard variety $Pic^0(X_{1})$. 

\paragraph{Group cohomology} 
If $\Gamma$ is a group, and $A$ is an abelian group,
we let $H^j(\Gamma,A)$ denote the
$j$-th cohomology group of $\Gamma$ with coefficients in $A$.
In particular, $H^1(\Gamma, A) = {\rm Hom}(\Gamma, A)$.
If $\phi: \Gamma \ra \Gamma'$ is a homomorphism of
groups there is a natural induced homomorphism 
of cohomology groups $\phi^*: H^j(\Gamma',A) \ra H^j(\Gamma,A)$.
(See \cite{Brown}, for a general reference.) We can formulate now:

\begin{lemma}
\label{lemma:cohomologyisnat}
There exists a natural isomorphism 
$$ Pic^0(X_{1})  \cong  H^1(\mathbb{Z}^n, \mathbb{R})/ H^1(\mathbb{Z}^n, \mathbb{Z})$$
such that,  for all $g \in G$,  the diagram
\begin{equation}
\label{vak_diagram_2}
\xymatrix@1{
Pic^0(X_{1}) \ar@1{->}[r]^(0.35){\cong} \ar[d]^{g^*} & H^1(\mathbb{Z}^n, \mathbb{R})/ H^1(\mathbb{Z}^n, \mathbb{Z})   \ar[d]^{\mu(g)^*} \\ 
Pic^0(X_{1}) \ar@1{->}[r]^(0.35){\cong}           &  H^1(\mathbb{Z}^n, \mathbb{R})/ H^1(\mathbb{Z}^n, \mathbb{Z})   } 
\end{equation}commutes.
\end{lemma}

Before proving the Lemma, we embark on the following 

\paragraph{Topological digression}
For the following, see \cite[Prop 11.4, Thm. 11.5]{maclane}.
Let $X$ be a (reasonable)  topological space, and $\Gamma$ a group
which acts properly discontinuously on $X$.
If $X$ is acyclic up to dimension $k$,  that is,  $$H_{j}(X) =  \begin{cases} \bbZ \;,  j = 0 \\ 0 \:\; , k \geq j \geq 1 \end{cases}, $$ then, for all $k \geq j \geq 1$,
there exist natural isomorphisms
\begin{equation*}
H^j_{sing}(X/\Gamma,A) \cong     H^j(\Gamma,A)
\end{equation*}
such that the diagram
\begin{equation}
 \label{topo_diagram_1}
\xymatrix@1{
H^j_{sing}(X/\Gamma,A) \ar@1{->}[r]^(0.5){\cong}\ar[d]^{f^*} & \ H^j(\Gamma,A)\ar[d]^{\phi^*}\\ 
H^j_{sing}(X/\Gamma,A) \ar@1{->}[r]^(0.5){\cong}           & \ H^j(\Gamma,A)      } 
\end{equation}
commutes, for all mappings $f: X \rightarrow X $,
and homomorphisms  $\phi : \Gamma \rightarrow \Gamma$ with 
\begin{equation*}
f( \gamma x) = (\phi \gamma)fx \; , \; \forall \gamma \in \Gamma, \; x \in X \; . 
\end{equation*} \

Now we give the proof of the above Lemma \ref{lemma:cohomologyisnat}.

\begin{proof}
Let $g \in G$ be a covering transformation of $p: X_{1} \ra X$, which is 
represented by $\gamma \in \Gamma$, acting on $\tilde{X}$, such that
the diagram
\begin{eqnarray}
\label{vak_diagram_1}
\xymatrix@1{
\tilde{X} \ar[r]^{\gamma} \ar[d]          & \ \tilde{X} \ar[d] \\
    X_{1} \ar[r]^{g}                &  X_{1} 
% X \ar[r]^{id} &  X 
}   % \; .  
\end{eqnarray}
is commutative. 
Then, for all $a \in \bbZ^n$, $x \in \tilde{X}$, the relation
\begin{eqnarray*}
\gamma a x =  \gamma a \gamma^{-1} \gamma x 
=  (\gamma a \gamma^{-1}) \gamma x 
=  \mu(g)(a) \gamma x 
\end{eqnarray*}
holds. Since $\tilde{X}$ is simply connected, by the theorem of Hurewicz, 
$H_1(\tilde{M})= 0$,  \cite{massey}.) Applying
 (\ref{topo_diagram_1}), it follows that there exists a natural 
 isomorphism $$H^1_{sing}(X_{1},A) \ra  H^1(\bbZ^n,A) \; ,$$
  such that the diagram 
 \begin{equation}
%  \label{topo_diagram_1}
\xymatrix@1{
 H^1_{sing}(X_{1},A) \ar@1{->}[r]^{\cong}\ar[d]^{g^*} & \ H^1(\bbZ^n,A)\ar[d]^{\mu(g)^*}\\ 
H^1_{sing}(X_{1},A) \ar@1{->}[r]^{\cong}           & \ H^1(\bbZ^n,A)      } 
\end{equation} 
 is commutative.
 Since the change of coefficients is natural, the stated commutative diagram
 \eqref{vak_diagram_2} is induced by composing with the projection 
 $H^{0,1}(X) \ra  H^1(X,\bbR)$, and the de Rham isomorphism 
$ H^1(X,\bbR) \ra   H^1_{sing}(X,\bbR) $. 
\end{proof}

\begin{proposition}
\label{vak_prop_2}
 Let $\Gamma$ be a virtually abelian K\"ahler group. Then $\Gamma$ admits a complex structure $J$. If\/ $\Gamma$ is projective then it also admits a polarisable complex structure $J'$.
\end{proposition}
\begin{proof}
Let $X$ be K\"ahler with $\pi_{1}(X) = \Gamma$. 
Note that  $H^1(\bbZ^n, \bbR)$ is isomorphic
to $\Hom(\bbZ^n, \bbR)$, which is just the dual of the module 
$\bbZ^n \otimes \bbR$.
Therefore, according to Lemma \ref{lemma:cohomologyisnat}, the
isomorphism
$H^1(X,\bbR)^*  \cong  (\bbZ^n \otimes \bbR)$, is a $G$-module
isomorphism, which is defined over $\bbQ$. 
Since $G$ acts by holomorphic transformations on $Pic^0(X_{1})$, the action
of $G$ on $H^1(X,\bbR)$ admits an invariant complex structure. 
In particular, by duality, $\bbZ^n \otimes \bbR$ attains a $G$-invariant
complex structure $J$. It thus defines a
complex structure for $\Gamma$. 
Since all identifications are defined over $\bbQ$, the complex 
structure $J$ is also polarisable if $X$ is projective. \end{proof}

%%%%%%%%%%%%%%%%%%%%%%%%%%%%%%%%%%%%%%%%%%%%
%%%%%%%%%%%%%%%%%%%%%%%%%%%%%%%%%%%%%%%%%%%%
\section{Construction of certain K\"ahler and projective manifolds}
\label{construction}
Let $\Gamma$ be a  virtually abelian group, which 
admits an invariant complex structure $J$. We construct
compact K\"ahler manifolds with fundamental group $\Gamma$.

\begin{lemma}
\label{lemma:vakg_emb}
There exists a homomorphism $\rho: \Gamma \ra \mathbb{C}^k \rtimes U(k)$,
which embeds the subgroup $\bbZ^n$ as a lattice 
$\Lambda = \rho(\bbZ^n)$ into $\bbC^k$. 
If,  in addition, $J$ is polarisable %admits a polarisation,
then there exists a polarisation for the lattice $\Lambda$.
\end{lemma}
\begin{proof}
We form 
the pushout $\bar{\Gamma}$ of \eqref{eq:fes}, with respect to the inclusion 
$\iota: \bbZ^n \hookrightarrow \bbR^n$.
This gives a homomorphism of exact sequences: 
\begin{equation} \label{eq:extequiv}
\begin{CD}
1 @>>>  \bbZ^n @>>>  \Gamma @>>> G    @>>>  1\\
 @.  @VV{\iota}V  @VV{}V  @VV{\id_{G}}V \\
1 @>>>  \bbR^n  @>>> \bar \Gamma @>>>  G @>>>  1
\end{CD} \; \, \; \; \; ,  
\end{equation} \ \\
(Explicitly, the construction of $\bar \Gamma$ is as follows:
Since  $\Gamma$ is an extension of $\bbZ^n$ by $G$, it is isomorphic to a 
group of the form
\begin{eqnarray*}
(\mathbb{Z}^{2k} \times G, \star); & (z_1,g_1) \star (z_2, g_2) = (z_1 + \mu (g_1)z_2 + \eta(g_1, g_2), g_1 g_2) \; ,
\end{eqnarray*}
with $\eta \in Z^2_{\mu}(G,\mathbb{Z}^{2k})$ a two-cocycle with respect to
$\mu$. Then $\bar \Gamma$ is  constructed as 
\begin{eqnarray*}
\bar \Gamma := (\mathbb{R}^{2k} \times G, \star) ;  & (u_1,g_1) \star (u_2, g_2) = (u_1 + \mu(g_1) u_2 + \iota \, \eta(g_1, g_2),g_1 g_2 )\; .  \; \;)
\end{eqnarray*} 
Now  $H^{2}_{\mu}(G, \mathbb{C}^{k}) = 0$, since $G$ is finite (see \cite[Chapter III,\S 10]{Brown}). Therefore, the
second exact sequence splits, and there exists a homomorphism 
\begin{equation} \label{eq:splitting}
\begin{CD}
1 @>>>  \bbR^n  @>>> \bar \Gamma @>>>  G @>>>  1 \\
 @.  @VV{\id_{\bbR^n}}V  @VV{}V  @VV{\id_{G}}V \\
 1 @>>>  \bbR^n @>>>  \bbR^n \rtimes_{\mu} G @>>> G    @>>>  1
\end{CD} \; \, \; \; \; .  
\end{equation} \ \\
The semi-direct product $\bbR^n \rtimes_{\mu} G$ maps naturally
into $\bbR^n \rtimes \GL(n,\bbR)$, which, via the above, constructs 
a homomorphism $f: \Gamma \ra \bbR^n \rtimes \GL(n,\bbR)$. 
Since $\mu(G)$ has an invariant complex structure, there exists $A \in   \GL(n,\bbR)$ such that $\mu^A(G) \leq \GL(k,\bbC)$. Since $\mu^A(G)$ is finite, it admits an
invariant hermitian metric,  and we may as well assume that $\mu^A(G) \leq \U(k)$.
Thus $\rho(g) =  A f(g) A^{-1}$ defines the required homomorphism 
$$ \rho: \Gamma \ra \bbC^k \rtimes \U(k) \;  $$  and 
$$ \Lambda = \rho(\bbZ^n) = A(\bbZ^n)$$
 is a  lattice in $\bbC^k$. 
Assume  furthermore that $\omega$ is a polarisation for $J$. Then 
$$ A^* \omega \, (u,v) = \omega(A^{-1}u, A^{-1}v) $$ defines a polarisation 
for the lattice $\Lambda$ in $\bbC^k$.
\end{proof}
%In particular,  $\Gamma$ is mapped to a discrete, cocompact subgroup of  
%$\mathbb{C}^k \rtimes U(k)$ with finite kernel.  \\

Let $Y$ be a compact K\"ahler manifold with $\pi_1(Y)=G$. 
Such a manifold exists by \cite[Proposition 15]{serre}, in
fact, it may be chosen to be projective. Let $\tilde{Y}$ be the 
universal covering of $Y$, and put
\begin{eqnarray*}
\tilde{X}:=  \;  \mathbb{C}^k \times \tilde{Y} \; . 
\end{eqnarray*}
Let $h: \Gamma \ra G$ denote the quotient homomorphism. 
Since $G$ acts on $\tilde{Y}$ by covering transformations, 
we can form the diagonal action of $\Gamma$ on $\tilde X$: 
\begin{eqnarray*}
% \Gamma \times \tilde X &\to& \tilde X; \\ 
\gamma  \cdot  (u,y) &=& ( \rho(\gamma) u, h(\gamma) y)\; . 
\end{eqnarray*}
Note that $\rho(\Gamma)$ is a discrete subgroup of $\mathbb{C}^k \rtimes U(k)$ and acts with compact quotient on $\mathbb{C}^k$. In particular, 
the action of $\Gamma$ on $\tilde{X}$ is properly discontinuous. 
Moreover, $\ker \rho$ is a finite subgroup of $\Gamma$, which 
projects isomorphically onto  $\ker \mu$. Therefore the action of  $\Gamma$ 
on $\tilde{X}$ is also free, since $G$ acts freely on $\tilde Y$. 
Since  $\Gamma$ acts by holomorphic transformations on $\tilde{X}$, the space of orbits 
$$  X: = \;  (\mathbb{C}^k   \times \tilde{Y}) / \, \Gamma $$ is a compact complex
manifold. Moreover, $X$ is also K\"ahler, since $\Gamma$ acts by isometries of
the product metric on $\tilde{X}$.
Put $T= \bbC^k/ \Lambda$.
Then there is a sequence of holomorphic coverings 
\begin{eqnarray}
\label{vak_covering}
 \xymatrix@1{
\tilde{X}= \mathbb{C}^k  \times \tilde{Y} \ar[d]
\\
T  \times \tilde{Y} \ar[d] \\ 
X =  (T  \times \tilde{Y}) / G
} \; , 
\end{eqnarray}
and also a locally trivial holomorphic fibering
$$ T \lra X \lra Y   \;  \; \,  .  $$

% We summarise: 
We have

\begin{proposition} \label{prop:constructX}
The complex manifold $X$ is compact K\"ahler,  and $\pi_{1}(X) = \Gamma$. 
If furthermore $J$ admits a polarisation then $X$ is a projective manifold. 
\end{proposition}
\begin{proof}
We already established that $X$ is compact K\"ahler. If $J$ is polarised, then by
Lemma \ref{lemma:vakg_emb}, $T$ is an abelian variety. By  
\cite[Proposition 15]{serre}, we may choose $\tilde{Y}$ as a projective
manifold. Therefore, $T \times \tilde Y$ is projective. By the quotient lemma
(cf.\ Lemma \ref{lemma:fquotients} below),
$X$ is projective. 
\end{proof}

Remark, if $\Gamma$ is torsionfree then the homomorphism $$\Gamma \lra \bbC^k \rtimes U(k)$$
constructed above is injective, and moreover, $\Gamma$ acts freely on
$\bbC^k$.  Therefore, 
\begin{eqnarray}
\label{flat_covering}
 \xymatrix@1{
T   \ar[d] \\ 
M = \bbC^k/ \Gamma   
} 
\end{eqnarray} 
is a holomorphic covering of flat K\"ahler manifolds,  where
$\pi_{1}(M) = \Gamma$. If $J$ has a polarisation,
$T$ is an abelian variety, and thus by Lemma \ref{lemma:fquotients} the complex manifold 
$M = T/G$ is projective. This proves Corollary \ref{cor:flatKaehler}. 

\paragraph{The quotient lemma for K\"ahler manifolds and projective varieties}
In the following form the quotient Lemma is an immediate
application of Kodaira's projective embedding theorem. 
(For the algebraic geometry variant,  see section \ref{sect:fieldsofdef1}.)
\begin{lemma}
\label{lemma:fquotients}
Let $M,\bar{M}$ be compact complex manifolds and $p:\bar{M} \to M$ be a finite, holomorphic covering. Then $\bar{M}$ is projective, if and only if  $M$ is projective.
\end{lemma}
\begin{proof}
Assume that $\bar{M}$ is projective. Then 
there is a K\"ahler metric such that  $[\omega] \in H^2(\bar{M}, \mathbb{Q})$,
where $\omega$ is the K\"ahler form. Let $G:= Deck(\bar{M}, p)$ denote the decktransformation group of the
covering, which is a finite group. Now put 
$\bar{\theta}:= \sum_{g \in G} g^* \omega$, where $g^* \omega$
denotes the pull back of $\omega$ by the action of $g \in G$ on $\bar{M}$.
Then $\bar \theta$  is a $G$-invariant
two-form on $\bar{M}$, and it is also a K\"ahler form, 
since the covering transformations are holomorphic maps.  Its cohomology 
class $[\bar{\theta}] \in  H^2(\bar{M}, \mathbb{Q})$ is invariant by $G$.
Let $\theta$ be the unique two-form on $M$ which satisfies 
$p^*\theta = \bar{\theta}$, and put $ H^2(\bar{M}, \mathbb{\bbR})^G$ for 
the $G$-invariant cohomology classes. Recall (cf.\  \cite[Prop.\ 11.14]{maclane}) that the map 
$p^*: H^2(\bar{M}, \mathbb{\bbR})^G \cong H^2(M,\mathbb{\bbR})$ 
is an isomorphism. Since this map is also compatible with the rational structures, it follows that $[\theta] \in H^2(M, \mathbb{Q})$.
By Kodaira's theorem (cf.\ \cite[Chapter VI]{Wells}),
$M$ is a projective manifold.

For the converse, assume that $M$ is a projective manifold and let $[\theta] \in H^2(M,\mathbb{Q})$ be the K\"ahler class of a K\"ahler metric on $M$. Then $\omega= p^*\theta$ is a K\"ahler form on $\bar M$ with rational K\"ahler class. It follows that $\bar M$ is projective.
\end{proof}

%\begin{remark} For an alternative proof, see \cite{Shaf}.
%\end{remark}

\section{Deformation spaces of torus quotients and existence of polarisations} 
The local deformation space for a complex torus $$X= T^{k}$$
is represented by the space of all complex structures
 $$ {\mathfrak D} = \{ J: \bbR^{2k} \to \bbR^{2k} \mid J^2 = -\id \} = 
\GL(2k,\bbR) \big/ \GL(k,\bbC) \;  . $$
(Compare \cite[p.408ff]{KodairaSpencerI-II} or 
\cite[\S 2.1]{ShimizuUeno} for details.) 
Now let $$ X_{G}  = X/G$$ be a finite quotient, and
 $\mu: G \to  \GL(k,\bbC)$ be the holonomy homomorphism
 associated to the action of $G$ on $X$
(see section \ref{construction}). Then 
$$ {\mathfrak D}(G)  = \{ J \in {\mathfrak D} \mid  A J A^{-1} =J \text{, for all $A$ in $\mu(G)$} \} $$ describes the local deformation space for  $X_{G}$. \\

The main result of this section is: 

\begin{proposition} \label{prop:contdeformation}
Let $J \in  {\mathfrak D}(G)$. Then there exists a 
continuous deformation $\bar J = (J_{t}): I= [0,1] \to {\mathfrak D}(G)$, 
with $J_{0}= J$,  such that  $J_{1}$ is polarisable.  
\end{proposition}

Moreover, we will show in Proposition \ref{prop:density} below 
that the set $ {\mathfrak D}(G)_{p}$ of
polarisable complex structures is dense in $ {\mathfrak D}(G)$.
This proves Theorem \ref{thm:deformation} in the introduction.\\

We deduce
\begin{corollary} Let $X_{G} = T/G$ be a finite quotient manifold of a complex torus $T$. Then $X_{G}$ may be continuously deformed to a projective manifold. 
\end{corollary}   

For the proof of Proposition \ref{prop:contdeformation},  we will study
the  cone of   K\"ahler forms $\Omega(G)$ which is associated to $G$
 and relate it to the deformation space ${\mathfrak D}(G)$. 
 
\subsection{K\"ahler forms and Siegel upper half space} 
Let $J: \bbR^{2k} \ra \bbR^{2k}$ be a complex structure. Recall that a non-degenerate two-form 
$\omega \in \bigwedge^2 (\bbR^{2k})^*$ is called compatible with $J$ (or $J$-hermitian), if  $$ \omega( J \cdot, J \cdot ) = \omega( \cdot,\cdot) \; . $$ 
Thus, equivalently, $\omega$
is $J$-hermitian if and only if $J \in \Sp(\omega)$.  If $\omega$ is  $J$-hermitian 
and the  associated
symmetric bilinear form $g_{\omega,J} = \omega(J \cdot, \cdot)$  
is positive definite then $\omega$ is called a \emph{K\"ahler form} for $J$. 

\subsubsection{Complex structures with common K\"ahler form}
We define  
$${\mathfrak D}^+_{\omega}=   \{ J \in {\mathfrak D} \mid  \text{$\omega$ is a K\"ahler form for $J$}  \},  $$
to be the set of all complex structures which admit $\omega$ as K\"ahler form.
Note that ${\mathfrak D}^+_{\omega}$ is non-empty and $\Sp(\omega)$ acts on ${\mathfrak D}^+_{\omega}$ by conjugation.
In fact, this  action is transitive: 
%a description of the set of  all complex structures $J$ which have
%K\"ahler form $\omega$, as follows:

\begin{proposition}
\label{Jswithomega} 
Let $J_{0} \in {\mathfrak D}^+_{\omega}$ be a complex structure which has K\"ahler form $\omega$, and let $J \in \frakD$. Then $J \in {\mathfrak D}^+_{\omega}$\ if and only if there exists $A \in \Sp(\omega)$ such that $J= A J_{0} A^{-1}$.
\end{proposition}

This is a consequence of Lemma \ref{Jconjugacy} below. We obtain 

\begin{corollary}
\label{moduli_siegel}
Let $J \in {\mathfrak D}^+_{\omega}$. 
There is a bijective correspondence 
\begin{eqnarray*}
\xymatrix@1{
 \mathfrak{S} :=  \; \Sp(\omega) \big/ \left(\Sp(\omega) \cap\GL(J)\right)  \ar@1{<->}[d]^{~} 
\\          \mathfrak{D}^+_{\omega} %:=  \,   \{ J \; | \; \omega_D\;  \text{is  a polarisation for J} \}
} 
\end{eqnarray*}
\end{corollary}
\begin{proof}
The claimed correspondence is  established  by the map 
\begin{eqnarray}
 [A] \in \mathfrak{S} \to A J  A^{-1} \; .
\end{eqnarray}
The correspondence is bijective by Proposition \ref{Jswithomega}.
\end{proof}

Note that $\Sp(\omega) \cap \GL(J)$ is the unitary group of the associated hermitian form $g_{\omega,J}$.  It is a maximal compact subgroup of  $\Sp(\omega)$, since $J \in  {\mathfrak D}^+_{\omega}$. 
The space $ \mathfrak{S}$ is customarily called the Siegel upper half plane,
and it is a Riemannian symmetric space in the sense of \'E.\ Cartan (cf.\ \cite{helgason}).

\subsubsection{Conjugacy classes in $\Sp(\omega)$}
Let $J \in \Sp(\omega)$. We put ${\rm sign}\, (\omega,J) $
for the signature of $g_{\omega,J}$. The signature 
classifies the conjugacy classes of complex structures  
$J \in \Sp(\omega)$.

\begin{lemma}
\label{Jconjugacy}
Let $J $, $J'$ be complex structures such that $J$ and  $J'$ are in 
$\Sp(\omega)$.  
Then ${\rm sign}\, (\omega,J) = {\rm sign} \, (\omega,J')$ if and only if
there exists $A \in \Sp(\omega)$ such that $J'= A J A^{-1}$.
\end{lemma}
\begin{proof} Clearly, ${\rm sign}\, (\omega,J) = {\rm sign} \, (\omega,J')$ if $J'= A J A^{-1}$, for some $A \in \Sp(\omega)$.
We show now that $J$ and 
$J'$ are conjugate in $\Sp(\omega)$ if  ${\rm sign}\, (\omega,J) = {\rm sign} \, (\omega,J')$.
We choose $B \in \GL(2k,\mathbb{R})$ such that $J'= BJB ^{-1}$. Then  
$$ ({B^{-1}})^* \omega \, ( \cdot, \cdot): = \omega(B^{-1} \cdot ,B^{-1} \cdot)$$ 
defines a compatible sympletic form for $J'$. Since  
$$ g_{  ({B^{-1}})^* \omega ,J'} = 
({B^{-1}})^* g_{\omega,J} \; , $$ 
we have ${\rm sign}\, (J',  ({B^{-1}})^* \omega ) = {\rm sign} \, (J',\omega)$. By the existence of a unitary basis, all hermitian symmetric forms for $J'$ of the same signature are equivalent under a $J'$-linear transformation. Hence,  there exists $ \tilde{B} \in \GL(2k,J')$, such that $ ({B^{-1}})^* \omega= {\tilde{B}}^*\omega$.
Now we have  $$  ({\tilde{B} B})^*\omega =  B^* (\tilde{B}^* \omega) =   B^*(  ({B^{-1}})^* \omega) = \omega  \; . $$
%\begin{eqnarray*}
% \omega(\tilde{g}gx, \tilde{g} g y) =  \omega_{\tilde{g}}(gx, gy)=  \omega_{g^{-1}}(gx,gy)% =   \omega(g^{-1}gx, g^{-1}gy)
%  = \omega(x,y) \; .
%\end{eqnarray*}
Thus, $\tilde{B}B \in \Sp(\omega)$. Since
$J'= \tilde{B} J' \tilde{B}^{-1} = \tilde{B} B J (\tilde{B} B)^{-1}$,  our claim follows.
\end{proof} % \hspace{1cm}

\subsection{The cone of  K\"ahler forms with respect to $G$}
Let $\Omega \subset \bigwedge^2 (\bbR^{2k})^*$ be the set of 
all non-degenerate two forms. Then $\Omega$ is an open cone in the 
vector space $\bigwedge^2 (\bbR^{2k})^*$. 
For a subgroup $G \leq \GL(2k,\bbR)$, we let  
$ \left(\bigwedge^2 (\bbR^{2k})^*\right)^G$
denote the vector space of $G$-invariant two forms. We then 
define % the K\"ahler cone with respect to $G$ as
$$ \Omega(G) =  \; \Omega \, \cap \,\left({\bigwedge}^2 (\bbR^{2k})^*\right)^G \; . $$

Let $J \in \frakD$ be a complex structure. 
If $G$ is a finite (or compact) subgroup of $\GL(J)$, there always
exists a non-degenerate $G$-invariant two-form:
%  In fact, this is a consequence of 

\begin{lemma} \label{lemma:Kaehlercones1}
Let $G \leq \GL(J)$ be a finite subgroup. Then
there exists a $G$-invariant K\"ahler form $\omega$ for $J$. 
\end{lemma}
\begin{proof}  Let $\omega'$ be a  K\"ahler form for $J$. 
Since $G$ is finite, we may find (by averaging $\omega'$ over $G$) 
a $G$-invariant K\"ahler form $\omega$.
\end{proof}
\noindent In particular,  if $G$ is finite the cone $\Omega(G)$ 
% $$ \Omega(G) \subset  \left(\bigwedge^2 (\bbR^{2k})^*\right)^G $$ 
is a non-empty and open subset in the vector space $\left(\bigwedge^2 (\bbR^{2k})^*\right)^G$. \\

% We also note

\begin{lemma} \label{lemma:Kaehlercones2}
Let $\omega \in \Omega(G)$. If  $G$ is finite then
there exists a $G$-invariant complex structure $J  \in {\mathfrak D}(G)$
which has K\"ahler form $\omega$.
\end{lemma}
\begin{proof}  By our assumption $G \leq \Sp(\omega)$.  Then $G$ acts by conjugation on the space ${\mathfrak S}={\mathfrak D}(G)_{\omega}$ of complex structures with K\"ahler form $\omega$, as is described in
 Corollary \ref{moduli_siegel}. By Cartan's theorem
(see \S \ref{Cartan}), $G$ has a fixed point $J$ in  ${\mathfrak D}(G)_{\omega}$.
\end{proof}
We call $\Omega(G)$ the cone of K\"ahler forms for $G$.

%%%
% Let $\ell = \dim_{\bbR}  (\bigwedge^2 (\bbR^{2k})^*)^G$.
%The $\bbQ$-vector space of rational  solutions 
%$(\bigwedge^2 (\bbQ^{2k})^*) )^{G}$ then satisifes
%$(\bigwedge^2 (\bbR^{2k})^*)^G = (\bigwedge^2 (\bbQ^{2k})^*) )^{G} \tensor \bbR$. 
%Now consider the non-empty open cone 
%$$ \Omega(G) =  \Omega \cap (\bigwedge^2 (\bbR^{2k})^*)^G 
%\subset (\bigwedge^2 (\bbR^{2k})^*)^G \; . $$
%%
\subsubsection{Rational points in $\Omega(G)$} \label{rationalpoints}
Next let
 $$\Omega(G)(\bbQ) = \; \Omega(G) \, \cap \, {\bigwedge}^2 (\bbQ^{2k})^*$$
denote the set of rational points in $\Omega(G)$.  
We remark
\begin{lemma} \label{lemma:rationaldense}
Let $G \leq \GL(2k,\bbQ) \cap \GL(J)$ be a finite subgroup. 
Then the set of rational points 
$\Omega(G)(\bbQ)$ is dense in $\Omega(G)$. 
In particular, $\Omega(G)(\bbQ)$ is non-empty. 
\end{lemma}
\begin{proof}
Since $G \leq \GL(2k,\bbZ)$,  
% which is the case if $G$ arises as the image of a holonomy representation. 
the vector subspace $ \left({\bigwedge}^2 (\bbR^{2k})^*\right)^G $ of  $\bigwedge^2(\bbR^{2k})^*$ is defined by rational equations, and it has a $\bbQ$-structure given by
the vector space of rational  solutions $\left(\bigwedge^2 (\bbQ^{2k})^*\right)^{G}$.
\end{proof}
This proves, in particular:

\begin{proposition} \label{prop:symplectic}
Let $G \leq \GL(2k,\bbZ)$ be a finite subgroup, which has
an invariant complex structure $J$. Then there exists $A \in \GL(2k,\bbQ)$
such that $A \, G \, A^{-1} \leq \Sp(2k,\bbQ)$. 
\end{proposition}

We are ready now for the 
\begin{proof}[Proof of Theorem \ref{thm:algequivalence}]
The equivalence of 1.\ and 3.\ is a consequence of
Theorem \ref{thm:deformation}. The implication from 1. \ to 
2.\ is implied by the previous Proposition \ref{prop:symplectic},
and its converse by Lemma \ref{lemma:Kaehlercones2}. 
\end{proof} 

\paragraph{K\"ahler cones and polarisations} 
Let $H^{1,1}(J,\bbR)$ denote the vector subspace  of $J$-hermitian
forms in $\bigwedge^2 (\bbR^{2k})^*$. The set of K\"ahler 
forms $\kappa(J)$ for $J$ is an open convex cone in 
 $H^{1,1}(J,\bbR)$, which is called the K\"ahler 
 cone for $J$. Consider   
 $$ H^{1,1}(J,\bbR)^G = H^{1,1}(J,\bbR) \cap 
  \left({\bigwedge}^2 (\bbR^{2k})^*\right)^G \; $$ 
 we then call  $$ \kappa(J,G) = \kappa(J) \cap
H^{1,1}(J,\bbR)^G$$  the K\"ahler cone for $J$ and $G$. 
(It is, in fact, the K\"ahler cone $\kappa(X_{G})$ of $X_{G}$.)
Note that the complex structure $J$ admits a polarisation if and
only if $\kappa(J,G)$ contains a rational point,
that is, if $$ \kappa(J,G) \cap {\bigwedge}^2 (\bbQ^{2k})^* \,  \neq \emptyset  \; .$$
By Lemma \ref{lemma:Kaehlercones2}, we have that
$$ \Omega(G) = \bigcup_{J \in {\mathfrak D}(G)} \kappa(J,G) \; . $$
Therefore, we may deduce from Lemma \ref{lemma:rationaldense}: 
\begin{corollary} Let $G \leq \GL(J)$ be a finite subgroup. 
Then there exists $J' \in {\mathfrak D}(G)$
such that $\kappa(J',G)$ has a rational point.  
\end{corollary}
\noindent In the remainder of this section, we shall show that such $J'$ may 
be obtained by a continuous deformation in ${\mathfrak D}(G)$  
starting from $J$.

\subsection{Lifting of curves from $\Omega(G)$ to $\frakD(G)$}
%%%
Let $J \in \frakD$ be a complex structure, and let $\omega \in \Omega$
be a K\"ahler form for $J$. We may lift curves in $\Omega$ starting
in $\omega$ to curves in $\frakD(G)$, starting in $J$:

\begin{lemma}  \label{lemma:Kaehlerlifting1}
Let $\bar \omega = (\omega_{t}): I \ra \Omega$ be a continuous curve
of non-degenerate two-forms, and let $J \in {\mathfrak D}^+_{\omega_{0}}$.
Then there exists a continuous 
curve $\bar J= (J_{t}): I \ra \frakD$, with $J_{0}=J$,
such that $\omega_{t}$ is a 
K\"ahler form for $J_{t}$. 
\end{lemma} 
\begin{proof}  Note that the transitive action of $\GL(2k,\bbR)$ on $\Omega$
gives an identification $$ \Omega = \GL(2k,\bbR) / \Sp(\omega_{0}) \; .$$
We may thus consider a lift (horizontal lift of the corresponding $ \Sp(\omega_{0})$-principal bundle) of $\bar \omega$, to obtain a continuous 
curve $\bar A =(A_{t}): I \ra \GL(2k,\bbR)$ such that 
$$ \omega_{t} =(A_{t})^* \omega_{0}. $$  Now let 
$J_{0}$ be any complex structure, which
has $\omega_{0}$ as a K\"ahler form. Then 
$$J_{t} = A_{t}^{-1}J_{0} A_{t}$$
defines the desired lift of $\omega_{t}$  to complex structures. 
\end{proof}

We now provide a lifting construction for curves in $\Omega(G)$.

\begin{lemma} \label{lemma:Kaehlerlifting2}
Let $\bar \omega = (\omega_{t}): I \ra \Omega(G)$ be a continuous curve
of non-degenerate $G$-invariant two-forms, and let $J \in \frakD(G)$, such
that $\omega_{0}$ is a K\"ahler form for $J$. Then there exists a continuous 
curve $\bar J= (J_{t}): I \ra \frakD(G)$ of $G$-invariant complex structures, 
with $J_{0} =J$, such that $\omega_{t}$ is a K\"ahler form for $J_{t}$. 
\end{lemma} 
\begin{proof} Write $\omega_{t} = {A_{t}}^* \omega_{0}$, as in the proof
of Lemma \ref{lemma:Kaehlerlifting1}. Then consider $$ \mu_{t}: I \ra \Hom(G, \Sp(\omega_{0})) \; , $$ which is given by $\mu_{t}(g) = A_{t} g A_{t}^{-1}$, $g \in G$.
Note here,  since 
 $$G \subset \Sp(\omega_{t}) = 
  A_{t}^{-1}\,  \Sp(\omega_{0}) A_{t} ,$$
 we have that 
 $$\mu_{t}(G) \subset \Sp(\omega_{0}) \; .$$ 
 Now let $$ \bar C= (C_{t}): I \ra \frakD^+_{\omega_{0}}$$ 
 be the continuous curve of barycenters for $\mu_{t}(G) J$, with
 $C_{0}= J$, as is constructed 
 in Proposition \ref{prop:Cartancontinuous}. The corresponding 
 complex structures $C_{t}$ are $\mu_{t}(G)$-invariant, and have
 K\"ahler form $\omega_{0}$. Then it is clear that $$J_{t} = A_{t}^{-1} C_{t} A_{t}$$
is a family of complex structures, such that $J_{t} \in \frakD(G)$ and
$\omega_{t}$ is a K\"ahler form for $J_{t}$. (Remark that $J_{t}$ is actually
the $G$-barycenter of $A_{t}^{-1} J_{0} A_{t}$ in 
${ \mathfrak D}_{\omega_{t}}$.)
\end{proof}

\subsubsection{Deformation to polarisable complex  structures}

%This proves Proposition \ref{prop:contdeformation}, as follows: 
We are now ready for the 

\begin{proof}[Proof of Proposition \ref{prop:contdeformation}] 
Choose $\omega_{0} \in \Omega(G)$, such that $\omega_{0}$
is a K\"ahler form for $J$. Let $\omega  \in \Omega(G)(\bbQ)$
be a rational two-form near $\omega_{0} $, and
$\bar \omega = (\omega_{t}): I \ra \Omega(G)$
a continuous path in $\Omega(G)$, such that
$\omega_{1} = \omega$. Then define $\bar J =(J_{t})$ to be the lift of
$\bar \omega$, with $J_{0} = J$, as in Lemma \ref{lemma:Kaehlerlifting2}.  
\end{proof}

\subsubsection{Density of polarisable complex  structures}
We may slightly refine the proof of Proposition \ref{prop:contdeformation} to obtain:
\begin{proposition} \label{prop:density} 
The set of polarisable complex structures
 ${\mathfrak D}(G)_{p}$ is dense in ${\mathfrak D}(G)$.
\end{proposition} 
\begin{proof}  Let $J_{0} \in {\mathfrak D}(G)$ with K\"ahler form $\omega_{0}$. 
 Let $\ {\mathfrak D}(G) \cap U$ be a neighbourhood of $J_{0}$, where $U$ is open in $ {\mathfrak D}$. We show that  ${\mathfrak D}(G)_{p} \cap U \neq \emptyset$ as follows:\\
 
 i) Consider the action of $\GL(n,\bbR)$ on  $ {\mathfrak D}$ by conjugation. Choose neighbourhoods $W$ of $1 \in \GL(n,\bbR)$ and $U_{1} \subset {\mathfrak D}$ of $J_{0} $ such that $W U_{1} \subset U$. \\

ii) By continuity of the symmetric space distance function $d$ on  $ {\mathfrak D}^+_{\omega_{0}}$ (cf.\  \S \ref{Cartan}), choose a metric ball $U_{\epsilon}= B_{\epsilon}(J_{0}) \subset  {\mathfrak D}^+_{\omega_{0}}$ with diameter $r$ such that the ball $B_{r+\epsilon}(J_{0})$ is contained in $U_{1}$. \\

 iii) For $g \in G$, consider the map $\pi_{g}:  \GL(n,\bbR) \ra {\mathfrak D}$,
 $A \mapsto A g A^{-1} J_{0} A g^{-1} A^{-1}$. Since $J_{0} \in   {\mathfrak D}(G)$,
 $\pi_{g}(1) = J_{0}$.  Choose an open neighbourhood $W_{\epsilon} \subset W$ of
 $1 \in \GL(n,\bbR)$ such that, for all $g \in G$,  
 $\pi_{{g}}(W_{\epsilon}) \cap  {\mathfrak D}^+_{\omega_{0}}
 \subset U_{\epsilon}$. \\
 
 iv) Consider the quotient map $\pi_{{\Omega}}:  \GL(n,\bbR) \ra {\Omega}$, where
 $\pi_{{\Omega}}(1) = \omega_{0}$. Since  $\pi_{{\Omega}}$ is locally a projection, we may assume that $W_{\epsilon} = Z \times V_{\epsilon}$, where $V_{\epsilon} \subset \Omega$ is a neighbourhood of $\omega_{0}$ and $Z$ is a neighbourhood
of the identity in $\Sp(\omega_{0})$. \\

Now let $\omega_{1} \in V_{\epsilon} \cap \Omega(G)(\bbQ)$ be a rational form, and join $\omega_{0}$ and $\omega_{1}$ by a path $(\omega_{t}): I \ra 
 \Omega(G) \cap V_{\epsilon}$. Using iv), we have a lift $(A_{t}): I \ra V_{\epsilon} \subset W_{\epsilon}$. Let  $\mu_{t}(G)  = A_{t} G A_{t}^{-1}$. Then the orbit $\mu_{t}(G) J_{0}$ 
is contained in ${\mathfrak D}^+_{\omega_{0}}$. It follows by iii) that $\mu_{t}(G) J_{0} \subset U_{\epsilon}$.  Using the remark following Proposition \ref{prop:Cartancontinuous} it follows from ii) that the associated curve of barycenters $\bar C$ is contained in $U_{1}$. Moreover, since $A_{t} \in W$, it follows by i) that the curve of $G$-invariant complex structures $J_{t}: I \ra {\mathfrak D}(G)$ is contained in $U$. Therefore, $J_{1} \in U \cap  {\mathfrak D}(G)$ is polarisable with
K\"ahler form $\omega_{1}$. 
\end{proof}

\subsection{The barycenters of a continuous deformation} \label{Cartan}
Let $G/K$ be a symmetric space, where $G$ is a real semisimple Lie group with finite center, and $K$ is a maximal compact subgroup. 
The following fact is due to \'E.\  Cartan (see \cite{helgason}[Chapter I,
Theorem 13.5]): Let $\mu \leq G$ be a finite (or compact) subgroup.
Then there exists a point $s \in G/K$ such that $\mu s = s$.\\

We shall need the following refinement of Cartan's result.

\begin{proposition} \label{prop:Cartancontinuous}
Let $\varphi_{t}: I \ra \Hom(\mu,G)$ be a continuous deformation,
where $\varphi_{0} (g) =g$, for all $g \in \mu$. Let $q  \in G/K$ be a fixed
point for $\mu$. Then there exists a canonical 
continuous curve $\bar q = (q_{t}): I \ra G/K$ 
with $\varphi_{t}(\mu) q_{t} = q_{t}$, and $q_{0}=q$.  
\end{proposition}

For the proof of Proposition \ref{prop:Cartancontinuous}, we 
need two elementary lemmata: 

\begin{lemma} \label{Lemma1}
Let $X$ be a compact metric space, and $f: I \times X \to \bbR$
a continuous function, where $I= [0,1]$. For $t \in I $, define 
$$ \bar f (t) = \min \{ f(t,q) \mid q \in X \} \, \; . $$
Then $\bar f: I \to \bbR$ is continuous. 
\end{lemma}
\begin{proof}  %We prove continuity at $0 \in I$. 
As an infimum of continuous functions, $\bar f$ is 
upper semi-\-continuous.  It remains to show that 
$\bar f$ is lower semi-continuous. That is, we show that
given any sequence $t_{m} \ra t$, $\liminf \bar f(t_m) \geq \bar f(t)$.
Assume to the contrary that $\lim \bar f(t_m) < \bar f(t)$, for
some sequence $t_{m} \ra t$. We use the fact, that there exists
a subsequence converging to $\liminf$. Since the minimum of $f(t_{m}, \cdot)$ is
assumed on $X$, there exist $q_{m} \in X$ such that
$f(t_{m},q_{m}) = \bar f(t_{m})$. By compactness of $X$, 
we may therefore also assume that $q_{m} \to p \in X$. 
By continuity of $f$, $\bar f(t_{m})=f(t_{m}, q_{m}) \to f(t,p) \geq \bar f(t)$. 
A contradiction.
\end{proof}

\begin{lemma}  \label{Lemma2} With the assumptions of Lemma \ref{Lemma1}, 
assume that the minimum $\bar f_{}(t)$ is attained at a unique point $q_{t} \in X$. 
Then the curve $t \to q_{t}$ is continuous. 
\end{lemma}
\begin{proof}  It is clearly enough to show that every sequence $t_{m} \ra t$, has 
a subsequence $s_{m}$ such that $q_{s_{m}} \to q_{t}$. By Lemma \ref{Lemma1}, 
$\bar f$ is continuous. Therefore, $ f(t_{m},q_{t_{m}})= \bar f(t_{m}) \to \bar f(t)= f(t,q_{t})$. Since $X$ is compact, there exists a subsequence $s_{m}$, such that $q_{s_{m}} \to p \in X$. By continuity of $f$, $f(t,p) = \bar f(t)$, and,  therefore, $f(t, \cdot)$ attains its minimum at $p$. Since, by our assumption, the minimum point is unique, we conclude that $p= q_{t}$. Therefore,  we have $q_{s_{m}} \to q_{t}$.  
\end{proof}

\begin{proof}[Proof of Proposition \ref{prop:Cartancontinuous}]
Let $d$ denote the distance function on the symmetric space $G/K$.
By \cite{helgason}[Chapter I,Theorem 13.5]), given $p_{0} \in X$,  
the  continuous function $$ f: X \ra \bbR \, , \,
f(q) = \sum_{g \in \mu} d^2(q, g p_{0})$$ attains its minimum at a 
\emph{unique} point $q_{0} \in G/K$
(the barycenter of the orbit $ \mu \, p_{0})$.  Moreover,  $q_{0} \in C$,
where $C$ is a compact ball around $p_{0}$ 
containing $\mu \, p_{0}$.  Therefore, 
the function 
$$ f(t, q) =  \sum_{g \in \mu} d^2(q, \varphi_{t}(g) p_{0})$$ satisfies the 
assumptions of the above two lemmata. Then the minimum point $q_{t}$ 
for $\bar f_{t}$ is  the barycenter for the orbit $ \varphi_{t}(\mu) p_{0}$,
and $q_{t}$ is a fixed point for $ \varphi_{t}(\mu)$. Thus, by Lemma \ref{Lemma2},
the curve of barycenters $t \to q_{t}$ is continuous. Choosing, $p_{0}=q$, 
we obtain $q_{0} = q$, as required.
\end{proof}

\begin{remark} The ball $C = C(p_{0})$, which appears in the proof may be chosen
as follows. Let $B_{r}(p_{0})$  be a metric ball of radius $r$, which contains the orbit $ \mu \, p_{0}$, and let $\ell$ denote its diameter.  Then $f(q) >  f(p_{0})$, for all $q \notin B_{r+\ell}(p_{0})$. Thus, we may take $C$ to be the closure of  
$B_{r+\ell}(p_{0})$. 
\end{remark}

\section{Fields of definition}

As we have seen, the set of polarisable complex structures 
$ \mathcal{D}_{p}(G)$ is dense in the space of all
$G$-invariant complex structures  $\mathcal{D}(G)$.
In the following, we explicitly construct points in 
$\mathcal{D}_{p}(G)$ which have additional symmetry. 
For this, we employ the approach used by Johnson in
\cite{Johnson} (also implicit in \cite{Serre2}). 
As an advantage, we can deduce information about the
fields of definition of the abelian varieties, which appear in
the constructions in section \ref{construction}. 
This leads us to the proof of Theorem \ref{thm:numberfields}

\subsection{Preliminaries} 
We collect some basic material from the theory of
abelian varieties. Our principal references are 
\cite{Lang} and \cite{Shimura}.

\paragraph{Field of definition for an abelian variety} \label{sect:fieldsofdef1}

Let  $T= \mathbb{C}^g /\Lambda$ be a complex torus, and   $$End(T):= \{f \in End(\mathbb{C}^g)\; | \; f \Lambda \subset \Lambda \}$$ 
 the endomorphism ring of $T$.
Recall that $T$ is called an abelian variety if it embeds holomorphically
as a subvariety of a projective space. 
An abelian variety defined over a subfield  $k$ of $\bbC$ 
is a  complex projective variety $X$ defined over $k$,
which has an algebraic group structure defined over $k$. 
If $T$ is biholomorphic to an abelian variety 
$X$ defined over $k$, then we will say that $T$ is defined over $k$.
Correspondingly, the endomorphism ring $End(T)$ of $T$ 
is said to be defined over $k$ if its elements 
correspond to  $k$-defined algebraic automorphisms of $X$. 

\paragraph{The quotient of a projective variety by a finite group}
Let $X$ be a complex projective variety defined over a subfield $k$ of $\bbC$.
We let $\Aut_k(X)$ denote the group of $k$-defined automorphisms 
of $X$.  
%We need the quotient lemma for projective varieties
%defined over $k$. 

\begin{lemma}
\label{lemma:algebraicquotient}
Let $Z$ be a complex projective variety defined over 
$k$ and let $G \subset
\Aut_k(Z)$ be a finite subgroup. Then the space of orbits
 $Z/G$ admits the structure of a $k$-defined projective variety,
 and $Z \to Z/G$ is % an \' etale covering 
 a finite morphism defined over $k$. 
\end{lemma}
\begin{proof}
%proof similar to Mumford's proof for the affine
%case as given in \cite[p. 67]{Mumford}.
For the standard  proof
see \cite[\S 13]{serre} or \cite[Ch. 4.3, Prop 16]{Shimura}.
%  or \cite[Thm. 3.1]{Lang}. 
We briefly sketch a more specialised argument
which works in the context of projective varieties.
%Let $K$ be a field and let $R$ be a graded $K$-algebra which is 
%finitely generated by homogeneous elements.
%For details see \cite[Ch. III.2; Exercise III-11]{Eisenbud}. 
% Moreover we can normalize this scheme in the following sense.
% Let $S^{(d)}:= \bigoplus_{\mu=1}^{\infty} S_{d \mu}$ for any $d  \in
% \mathbb{Z}$. One can show that $proj(S^{(d)})= proj(S)$ for all $d \in \mathbb{Z}$ and that moreover 
% there exists an $n \in \mathbb{Z}$, such that $S^{(n)}$ is a graded $k$ algebra generated 
% by homogeneous elements of degree $1$. $proj(S^n)$ is often called normal form
% for the scheme $proj(R)$. For details see.  
Let $R$ be the homogeneous coordinate ring of $Z$. Then
$R= R_k \otimes \mathbb{C}$, where $R_k$ is a finitely generated
graded $k$-algebra. According to Noether's theorem, 
the algebra of invariants  $R_k^G $ is a finitely generated graded $k$-algebra. 
It follows that 
$R^G = R_k^G \otimes \mathbb{C}$ is a finitely generated
$\mathbb{C}$-algebra with $k$-structure. 
To any such algebra $A$ one can assign a scheme, called $proj(A)$, such 
that $proj(A)$ is a projective variety defined over 
$k$. See  \cite[p. 282]{Mumford2}.
Then $Z/G= proj(R^G)$
% By using the normal form for $proj(R^G)$ we can
%substitute $spec(R^G)$ with $proj(R^G)$ in the result of Mumford \cite[p. 67]{Mumford}
is the desired quotient variety.
\end{proof}

Note that $\pi: Z \ra Z/G$ is a quotient in the categorical sense. This means, given
any $k$-defined morphism $f: Z \ra Y$, where $Y$ is a variety defined over $k$,
and $f$ is constant on $G$-orbits, there exists a unique $k$-defined morphism 
$\bar f: Z/G \ra Y$ such that $f = \bar f \pi$.

\paragraph{Abelian varieties with complex multiplication}
A number field $F$ is called a CM-field, if $F$ is a totally imaginary quadratic
extension of a totally real number field $E$. In particular,  we have
$[F:\mathbb{Q}] = 2g$ for  $g=[E: \mathbb{Q}]$ and $F= E(a)$ for a totally
imaginary $a \in F$
with $a^2 \in E$. A choice of $g$ non conjugate embeddings       
$\varphi_i:
F \hookrightarrow \mathbb{C}$, $i = 1, \ldots ,g$, is called 
a CM-type. It gives  rise to an  embedding  
$\varphi= (\varphi_i)_{i=1, \ldots ,g} : F \to \mathbb{C}^g$ which extends to
an isomorphism $\Phi: F  \otimes_{\mathbb{Q}} \mathbb{R} \to \mathbb{C}^g $.
%For details see \cite{Lang}[Ch. 1, \S 1], \cite{Birken}[] or
%\cite{Shimura}[Ch.2].

Let $T$ be a complex torus, and 
 $End_{\mathbb{Q}}(T):= End(T) \otimes \mathbb{Q}$. 
Then $T$
is said to have complex multiplication by a CM-field $F$, if there exists an embedding $\iota: F
\to
End_{\mathbb{Q}}(T)$ such that $2g= 2 dim (T)$. 
%For details see \cite[\S 6]{Birken},\cite[Ch. 5.2]{Shimura}.

For any  CM-field $F$ one can  construct an abelian variety
with complex multiplication by  $F$. 
Let $\Lambda$ 
be a lattice in $F$. 
Then $\Phi(\Lambda) \subset
\mathbb{C}^g$ is a lattice in
 $\mathbb{C}^g$ and $$ X_F:= \mathbb{C}^g
/\Phi(\Lambda) $$  is a complex torus with $F \subset
End_{\mathbb{Q}}(X_F)$. Since  $X_{F}$ admits a 
polarisation (see \cite[Thm.\ 4.1]{Lang} or \cite[Ch. 6]{Shimura}), 
% $X_{F}$ 
it is an abelian variety with complex multiplication by $F$.
%or \cite[\S 6]{Birken}.
Since $X_{F}$ has complex multiplication 
by a CM-field $F$, $X_{F}$ and $End(X_{F})$ may be defined over an algebraic number field $k$ (see \cite[Ch.5; Prop.\ 1.1]{Lang} or  \cite{Shimura}).

%A complex torus  $T$ is  called an abelian variety if $T$ is 
%biholomorphic to a complex projective variety. 
%variety $X$. 

\subsection{$\mathbb{Q}[G]$-modules with complex multiplication}
Let $G$ be a finite group and let $V$ be a finite dimensional  
$\mathbb{Q}[G]$-module. If there exists a homomorphism of 
rings $F \ra End_{\mathbb{Q}[G]}(V)$, where $F$ is a CM-field, 
we say that $V$ has complex multiplication by $F$.  

% For details see \cite[\S1; \S 3; Prop. 3.1]{Johnson}, \cite{Shimura2} and \cite{Albert}. 

%We want to construct out of such $\mathbb{Q}[G]$-modules an abelian variety $X$ with
%multiplication by $F$ with $G \subset
%End(X)$.

\begin{proposition}
\label{prop:specialJ}
Let $V$ be a $\mathbb{Q}[G]$-module, which has complex multiplication by $F$. 
Then there exists a complex structure $J_{F}$ on
$V \otimes_{\mathbb{Q}} \mathbb{R}$,  such that 
$F \subset End(V \otimes_\mathbb{Q} \mathbb{R},J_{F})$ and 
% $G \subset GL(V\otimes_{\mathbb{Q}} \mathbb{R}, J_{F})$. 
$J_{F} \in End_{\mathbb{R}[G]}(V\otimes_{\mathbb{Q}} \mathbb{R})$.
\end{proposition}
\begin{proof}
% and hence $V \otimes_{\mathbb{Q}} \mathbb{R} = \bigoplus_{i=1}^m
% F \otimes_{\mathbb{Q}} \mathbb{R}$. 
Choose a CM-type for $F$. By pulling back via
 $\Phi: F \otimes_{\mathbb{Q}} \mathbb{R} \to \mathbb{C}^g$ 
 we obtain a complex structure on $F \otimes_{\mathbb{Q}} \mathbb{R}$. 
%such that $F \subset End(F \otimes_\mathbb{Q} \mathbb{R},J_F)$. 
%where $J_0$ 
%is the standard complex structure on $\mathbb{R}^{2g}$ induced by $\mathbb{C}^g$.
Since $V$ is an $F$-module,  $V \cong F^m$, for some $m \in \mathbb{N}$.
We  let $J_{F}$ denote the product complex structure on $V \otimes_{\mathbb{Q}} \mathbb{R}$. 
Then $F \subset End(V \otimes_\mathbb{Q} \mathbb{R},J_{F})$.

Write $F = E(a)$, as above, and put $b = a^2 \in E$. 
%Then there exists  $a_i \in \mathbb{R}$, such that $\varphi_i(a)= i a_i$. 
%Then $b_i:=\varphi_i(b) = -a_i^2$. 
Let $E_i:= \varphi_i (E)$. Note that $b$ acts on  $F \otimes_{E_i} \mathbb{R}$ 
by scalar multiplication with $\varphi_i(b)$.
% and hence we get a 
% decomposition $F = \bigoplus_{i= 1}^g F \otimes_{E_i} \mathbb{R}$ and
 We have a $G$-invariant decomposition $V \otimes_{\mathbb{Q}} \mathbb{R} = \bigoplus_{i=1}^g V \otimes_{E_i}
\mathbb{R}$. With respect to this
decomposition $J_{F}$ acts as 
multiplication with $a \otimes \frac{1}{\sqrt{\varphi_i(b)}}$ on each factor. Since $a
\in End_{\mathbb{Q}[G]}(V)$, it follows that  $J_{F}$ is $G$-invariant. 
\end{proof}

If $V$ is an irreducible ${\mathbb{Q}[G]}$-module 
then $D= End_{\mathbb{Q}[G]}(V) $ is a finite
dimensional division algebra. Moreover, this division algebra
admits a  positive involution \cite[Proof of Prop. 3.2]{Johnson}. By a
classification of Albert either $D$  contains a
CM-field $F$ or $D$ is a totally real division algebra. 
(See \cite[\S1; \S 3; Prop. 3.1]{Johnson}, \cite{Shimura2} and \cite{Albert}). 
In the latter case, $End_{\mathbb{Q}[G]}(V \oplus V)$ has complex 
multiplication by a CM-field $F$. This has the following consequence: 

\begin{proposition}
\label{prop:CMdecomposition}
Let $V$ be a $\mathbb{Q}[G]$-module such that 
$V \tensor_{\bbQ} \bbR$ admits a $G$-invariant complex structure. 
Then there exists a decomposition of $V$ into $G$-submodules $W_{j}$,
such that $W_{j}$ admits complex multiplication by a CM-field $F_{j}$.
\end{proposition}
\begin{proof}
In fact, $V = \bigoplus V_i$ where
$V_i$ are irreducible $\mathbb{Q}[G]$-modules. 
Since $V \otimes_{\mathbb{Q}} \mathbb{R}$ admits a 
$G$-invariant complex structure $J$,
% the multiplicity of $G$-submodules in $V \otimes_{\mathbb{Q}} \mathbb{R}$ 
% with totally centraliseris even. This implies that also t
the multiplicity of the $V_i$ with totally real centraliser 
$End_{\mathbb{Q}[G]}(V_i)$ is even in $V$.
%Johnson proved, that the multiplicity of the modules $V_i$ 
%with $D_i$ a totally real division algebra  is even. See \cite[Thm.3.3]{Johnson}. 
\end{proof}

\subsection{Kummer varieties}
Let $\Lambda \subset V$ be a lattice, and $J$ a complex structure
on $V \tensor_{\bbQ} \bbR$. We denote by 
$$ T_{\Lambda,J}:= (V\otimes_{\mathbb{Q}} \mathbb{R},J)/ \Lambda \; $$
the associated complex torus. 
%The latter, of course, means  that the endomorphisms of $T$ 
%induce $k$-defined algebraic automorphisms of the associated projective 
%variety $X$. 

% $\Phi (\tilde{\Lambda}_i) = X_i$ are isomorphic (in the category of abelian varieties) where 
%the $X_i$ are abelian varieties with the property that  $End(X_i)$ and $X_i$ are defined over an algebraic number field. 
%Hence $X_{\tilde{\Lambda}} = \bigoplus X_i$ with $\bigoplus X_i$ and $End(\bigoplus X_i)$ are defined 
%over an algebraic number field $k$.

%

\begin{proposition}
\label{prop:geometricconsequence1}
Let $V$ be a $\mathbb{Q}[G]$-module, which has complex multiplication by $F$.  
Let $\Lambda \subset V$ be a $G$-invariant lattice.
Then the associated complex torus $T_{\Lambda,J_{F}}$ and its endomorphism ring $End(T_{\Lambda,J_{F}})$ may be defined over an algebraic number field $k$.
In particular, the induced linear action of $G$ on $T_{\Lambda,J_{F}}$ is defined
over $k$.  
\end{proposition}
\begin{proof}
% 
%Let $2g = [F: \mathbb{Q}]$. By prop.\ref{prop:specialJ} there exists a $G$-invariant complex structure 
%$J$ on $V \otimes_{\mathbb{Q}} \mathbb{R}$.
%As in the last Proposition $V \otimes_{\mathbb{Q}} \mathbb{R} = \bigoplus_{j=1}^m F \otimes_\mathbb{Q} \mathbb{R}$,
%for some $m \in \mathbb{N}$, is a $J$ invariant decomposition.
As in the proof of Proposition \ref{prop:specialJ}, write $V$ as
a sum of 1-dimensional $F$-vector subspaces $F_{j}$. By construction,
$F_{j} \tensor_{ \bbQ }\bbR$ is a $J_{F}$-subspace of $V \tensor_{ \bbQ} \bbR$.
Let $\Lambda_j = \Lambda  \cap F_{j}$.
Then $\tilde{\Lambda}= \bigoplus_{j=1}^m \Lambda_j$ is a lattice in $V$, 
which is contained in $\Lambda$. Define $X_{j} = T_{\Lambda_{j},J_{F}}$.
Then $X_{j}$ is an abelian variety with complex multiplication by $F$, and
$$ X_{\tilde{\Lambda}} =  \bigoplus_{j =1}^m X_{j}$$
decomposes as a direct product of abelian varieties, which are defined over
an algebraic number field. 
Therefore, $X_{\tilde{\Lambda}}$ and its endomorphism ring $End(X)_{\tilde{\Lambda}}$ 
may be defined over an algebraic number field $k_{0}$.

% Since  $\Lambda$ and $\tilde{\Lambda}$ are commensurable,
Consider the covering 
$X_{\tilde{\Lambda}} \to X_{\Lambda}$. Since the elements of the finite kernel are 
algebraic over $\bbQ^{a}$, $X_{\Lambda}$ and the covering map may be defined 
over an algebraic number field $k$ (using the quotient lemma),  
$k$ containing $k_{0}$.
Since $G \subset End_{\mathbb{Q}}(X_{\tilde{\Lambda}})$, 
there exists an $\ell \in \mathbb{N}$, such that $\ell g \in End(X_{\tilde{\Lambda}})$,
and, therefore, the linear action of $\ell g$ on $X_{\tilde{\Lambda}}$ is defined over 
$k_{0}$.
Hence we get a diagram
\begin{eqnarray*}
\xymatrix@1{
X_{\tilde{\Lambda}} \ar[r]^{\ell g} \ar[d] & X_{\tilde{\Lambda}} \ar[d] \\
X_{\Lambda} \ar[r]^{\ell g}  & X_{\Lambda}
} \; ,
\end{eqnarray*}
where the upper and the downward maps are $k$-defined. By the universal 
property of the quotient,  
$\ell g \in End_k(X_{\Lambda})$. By the Galois-criterion for rationality, 
it follows that also $g \in End_k(X_{\Lambda})$.
%$G \subset End_k(X')$.
\end{proof}

Thus, the following is a consequence of Proposition \ref{prop:CMdecomposition}.

\begin{proposition}
\label{prop:geometricconsequence2}
Let $V$ be a $\mathbb{Q}[G]$-module, such that $V \tensor_{\bbQ} \bbR$
admits a $G$-invariant complex structure. Let $\Lambda \leq V$ be
a $G$-invariant lattice. Then there exists a complex structure 
$J \in {\mathcal D}(G)_{p}$ such that the
complex torus $T_{\Lambda,J}$ and its endomorphism ring 
 $End(T_{\Lambda,J})$ may be defined over an algebraic number field $k$.
%In particular, the induced linear action of $G$ on $T_{\Lambda,J_{F}}$ is defined
%over $k$.  
\end{proposition}
Note, as follows from the quotient lemma, that the geometric quotient 
$$ K =   T_{\Lambda,J} /G $$ 
attains the structure of a $k$-defined projective variety.
It is called a Kummer-variety for $G$.\\

Combining the last result  with the proof of 
Proposition \ref{prop:constructX} we arrive at:
\begin{proposition}
Let $\Gamma$ be a virtually abelian K\"ahler group. Then there exists a 
non-singular complex projective
variety $X$ defined over a number field $k$ with $\pi_1(X)=\Gamma$.
\end{proposition}
\begin{proof}
Let $\Gamma$ be such a K\"ahler group, and $\Lambda \cong \bbZ^k$ 
a normal subgroup with $G= \Gamma / \Lambda$. The characteristic
representation $\mu$ of $G$ turns $V = \Lambda \tensor_{\bbZ} \bbQ$ 
into a $\bbQ[G]$-module. Since $\Gamma$ is K\"ahler, $V \tensor_{\bbQ } \bbR$
has a $G$-invariant complex structure. We remark that $\Gamma$ can be mapped
homomorphically to $V \ltimes_{\mu} G$ (analogously as in \eqref{eq:extequiv} and
\eqref{eq:splitting}).

As shown in Proposition \ref{prop:constructX},
there exists a complex manifold $X$
with $\pi_1(X)= \Gamma$, such that 
\begin{eqnarray*}
\xymatrix@1{
\tilde{X}= \mathbb{C}^k  \times \tilde{Y} \ar[d]
\\
T  \times \tilde{Y} \ar[d] \\ 
X =  (T  \times \tilde{Y}) / G
} \; 
\end{eqnarray*}
are holomorphic coverings, where $T = \mathbb{C}^k / \Lambda$ is a complex torus.
Note that $G$ acts on $T$ by affine transformations. The linear part of the action is
given by the representation $\mu: G \ra End(T)$. Using Proposition \ref{prop:geometricconsequence2}, 
we may assume in the construction that $T = T_{\Lambda,J}$ 
is an abelian variety defined over an algebraic number field $k_{0}$, 
such that the linear parts of $G$ act by elements in $\Aut_{k_{0}}(T)$.
Moreover, as remarked above, since $\Gamma$ splits over $\bbQ$, 
the translation parts of the $G$-action may be assumed to 
be contained in the torsion group of  $T_{\Lambda,J}$. Therefore,
there exists an algebraic number field $k$,  such that $G$ act by elements
in $\Aut_{k}(T)$.

Note now that $\tilde{Y}$ is also a complex projective variety defined over $\bbQ$,
with $G \subset \Aut_{\bbQ}(\tilde{Y})$, see the proofs given in  
\cite{serre} or \cite{Shafa}.
%Now identify  $\mathbb{C}^g$ with $(\mathbb{R}^{2g},J_0)$ where $J_0$ is the standard complex structure.
%Then $(\mathbb{R}^{2g},J_0)$ is a $\mathbb{Q}[G]$-module, $J_0$ is $G$-invariant 
%and $\Lambda \subset \mathbb{R}^{2g}$ is a $G$-invariant lattice.
By Lemma \ref{lemma:algebraicquotient},
$X$ is a projective variety, which may be defined over the number field $k$.
\end{proof}

%%%%%%%%%%%%%%%%%%%

\addcontentsline{toc}{chapter}{Literaturverzeichnis}
\bibliographystyle{plain}
\pagestyle{empty}
\bibliography{kpgs}

\begin{thebibliography}{10}

\bibitem{Albert}
Abraham~Adrian Albert.
\newblock {\em Structure of algebras}.
\newblock Colloquium publications / American Mathematical Society ; 24.
  American Math. Soc., Providence, RI, 1964.

\bibitem{ABKT}
J.~Amor{\'o}s, M.~Burger, K.~Corlette, D.~Kotschick, and D.~Toledo.
\newblock {\em Fundamental groups of compact {K}\"ahler manifolds}, volume~44
  of {\em Mathematical Surveys and Monographs}.
\newblock American Mathematical Society, Providence, RI, 1996.

\bibitem{bauescortes}
Oliver Baues and Vicente Cort{\'e}s.
\newblock Aspherical {K}\"ahler manifolds with solvable fundamental group.
\newblock {\em Geom. Dedicata}, 122:215--229, 2006.

\bibitem{BGL}
Ch. Birkenhake, V.~Gonz{\'a}lez, and H.~Lange.
\newblock Automorphism groups of {$3$}-dimensional complex tori.
\newblock {\em J. Reine Angew. Math.}, 508:99--125, 1999.

\bibitem{birken}
Christina Birkenhake and Herbert Lange.
\newblock {\em Complex Abelian varieties}.
\newblock Springer Verlag, 2., augmented edition, 2004.

\bibitem{Brown}
Kenneth~S. Brown.
\newblock {\em Cohomology of groups}.
\newblock Springer Verlag, 1982.

\bibitem{fujiki}
Akira Fujiki.
\newblock Finite automorphism groups of complex tori of dimension two.
\newblock {\em Publ.RIMS,Kyoto Univ}, 24:1--97, 1988.

\bibitem{helgason}
Sigurdur Helgason.
\newblock {\em Differential geometry, Lie groups, and symmetric spaces}.
\newblock Academic Press, Inc., 1978.

\bibitem{Johnson}
F.E.A. Johnson.
\newblock {Flat algebraic manifolds.}
\newblock {Geometry of low-dimensional manifolds. 1: Gauge theory and algebraic
  surfaces, Proc. Symp., Durham/UK 1989, Lond. Math. Soc. Lect. Note Ser. 150,
  73-91 (1990)}, 1990.

\bibitem{KodairaSpencerI-II}
K.~Kodaira and D.~C. Spencer.
\newblock On deformations of complex analytic structures. {I}, {II}.
\newblock {\em Ann. of Math. (2)}, 67:328--466, 1958.

\bibitem{Lang}
Serge Lang.
\newblock {\em Complex multiplication}.
\newblock Die Grundlehren der mathematischen Wissenschaften in
  Einzeldarstellungen ; 255. Springer, Berlin, 1983.

\bibitem{maclane}
Saunders Mac~Lane.
\newblock {\em Homology}.
\newblock Springer Verlag, 1995.

\bibitem{massey}
William~Schumacher Massey.
\newblock {\em Singular homology theory}.
\newblock Springer Verlag, 1980.

\bibitem{mumford}
David Mumford.
\newblock {\em Abelian varieties}.
\newblock Oxford Univ. Pr., 2. ed. edition, 1974.

\bibitem{Mumford2}
David Mumford.
\newblock {\em The red book of varieties and schemes}.
\newblock Lecture notes in mathematics ; 1358. Springer, Berlin, 1988.

\bibitem{popovzarhin}
Vladimir~L. Popov and Yuri~G. Zarhin.
\newblock Finite linear groups, lattices, and products of elliptic curves.
\newblock {\em J. Algebra}, 305(1):562--576, 2006.

\bibitem{riesterer}
Johannes Riesterer.
\newblock {K\"ahlergruppen und endliche Erweiterungen von $\mathbb{Z}^n$}.
\newblock {Diploma thesis, Universit\"at Karlsruhe}, 2007.

\bibitem{serre}
Jean-Pierre Serre.
\newblock {Sur la topologie des vari\'et\'es alg\'ebriques en caract\'eristique
  $p$.}
\newblock {\em {Sympos. internac. Topolog\'{\i}a algebraica 24-53 (1958)}},
  1958.

\bibitem{Serre2}
Jean-Pierre Serre.
\newblock {Exemples de variet\'es projectives conjugu\'ees non hom\'eomorphes}.
\newblock {\em C.R.\ Acad.\ Sci.\ Paris}, 258:4194--4196, 1964.

\bibitem{Shafa}
I.~R. Shafarevich.
\newblock {\em Basic algebraic geometry}.
\newblock Springer-Verlag, New York, 1974.
\newblock Translated from the Russian by K. A. Hirsch, Die Grundlehren der
  mathematischen Wissenschaften, Band 213.

\bibitem{ShimizuUeno}
Yuji Shimizu and Kenji Ueno.
\newblock {\em Advances in moduli theory}, volume 206 of {\em Translations of
  Mathematical Monographs}.
\newblock American Mathematical Society, Providence, RI, 2002.
\newblock Translated from the 1999 Japanese original, Iwanami Series in Modern
  Mathematics.

\bibitem{Shimura2}
Goro Shimura.
\newblock On analytic families of polarized abelian varieties and automorphic
  functions.
\newblock {\em Ann. of Math.}, 78:149--192, 1963.

\bibitem{Shimura}
Goro Shimura and Yutaka Taniyama.
\newblock {\em Complex multiplication of Abelian varieties and its applications
  to number theory}.
\newblock Publications / Mathematical Society of Japan ; 6. Math. Soc. of
  Japan, Tokyo, 1961.

\bibitem{voisin}
Claire Voisin.
\newblock {\em Hodge theory and complex algebraic geometry}.
\newblock Cambridge University Press, 2002.

\bibitem{Wells}
Raymond~O. Wells.
\newblock {\em Differential analysis on complex manifolds}.
\newblock Springer Verlag, 1980.

\bibitem{wolf}
Joseph~A. Wolf.
\newblock {\em Spaces of constant curvature}.
\newblock MacGraw-Hill, 1967.

\bibitem{zieschang}
Heiner Zieschang.
\newblock {\em Finite groups of mapping classes of surfaces}.
\newblock Springer Verlag, 1981.

\end{thebibliography}
\end{document}